\RequirePackage[l2tabu,orthodox]{nag}
\documentclass
[11pt,letterpaper]
{article}

\newcommand\MYcurrentlabel{xxx}
\newcommand{\MYstore}[2]{%
  \global\expandafter \def \csname MYMEMORY #1 \endcsname{#2}%
}
\newcommand{\MYload}[1]{%
  \csname MYMEMORY #1 \endcsname%
}
\newcommand{\MYnewlabel}[1]{%
  \renewcommand\MYcurrentlabel{#1}%
  \MYoldlabel{#1}%
}
\newcommand{\MYdummylabel}[1]{}
\newcommand{\torestate}[1]{%
  \let\MYoldlabel\label%
  \let\label\MYnewlabel%
  #1%
  \MYstore{\MYcurrentlabel}{#1}%
  \let\label\MYoldlabel%
}
\newcommand{\restatetheorem}[1]{%
  \let\MYoldlabel\label
  \let\label\MYdummylabel
  \begin{theorem*}[Restatement of \cref{#1}]
    \MYload{#1}
  \end{theorem*}
  \let\label\MYoldlabel
}
\newcommand{\restatecorollary}[1]{%
  \let\MYoldlabel\label
  \let\label\MYdummylabel
  \begin{corollary*}[Restatement of \cref{#1}]
    \MYload{#1}
  \end{corollary*}
  \let\label\MYoldlabel
}
\newcommand{\restatelemma}[1]{%
  \let\MYoldlabel\label
  \let\label\MYdummylabel
  \begin{lemma*}[Restatement of \cref{#1}]
    \MYload{#1}
  \end{lemma*}
  \let\label\MYoldlabel
}
\newcommand{\restateprop}[1]{%
  \let\MYoldlabel\label
  \let\label\MYdummylabel
  \begin{proposition*}[Restatement of \cref{#1}]
    \MYload{#1}
  \end{proposition*}
  \let\label\MYoldlabel
}
\newcommand{\restatefact}[1]{%
  \let\MYoldlabel\label
  \let\label\MYdummylabel
  \begin{fact*}[Restatement of \cref{#1}]
    \MYload{#1}
  \end{fact*}
  \let\label\MYoldlabel
}
\newcommand{\restatedefinition}[1]{%
  \let\MYoldlabel\label
  \let\label\MYdummylabel
  \begin{definition*}[Restatement of \cref{#1}]
    \MYload{#1}
  \end{definition*}
  \let\label\MYoldlabel
}
\newcommand{\restate}[1]{%
  \let\MYoldlabel\label
  \let\label\MYdummylabel
  \MYload{#1}
  \let\label\MYoldlabel
}

\newcommand{\Hoelder}{H\"{o}lder\xspace}
\newcommand{\Holder}{\Hoelder}

\usepackage[notes=true,later=false,camera=false]{dtrt}
\usepackage[utf8]{inputenc}
\usepackage{etex}
\usepackage{ stmaryrd }
\usepackage{xspace,enumerate}
\usepackage[T1]{fontenc}
\usepackage[full]{textcomp}
\usepackage[american]{babel}
\usepackage{mathtools}

\usepackage{amsthm}
\usepackage{empheq}

   \usepackage{hyperref}
\usepackage[capitalise,nameinlink]{cleveref}
\crefname{lemma}{Lemma}{Lemmas}
\crefname{fact}{Fact}{Facts}
\newcommand{\colorconstraints}{\text{Color Constraints}}
\crefname{colorconstraints}{(color constraints)}{Color Constraints}
\crefformat{colorconstraints}{#2\colorconstraints#3}
\crefname{indsetconstraints}{(indset constraints)}{IndSet Constraints}
\crefformat{indsetconstraints}{#2$\mathsf{IndSet\ Axioms}$#3}
\crefname{theorem}{Theorem}{Theorems}
\crefname{mtheorem}{Theorem}{Theorems}
\crefname{corollary}{Corollary}{Corollaries}
\crefname{claim}{Claim}{Claims}
\crefname{example}{Example}{Examples}
\crefname{problem}{Problem}{Problems}
\crefname{definition}{Definition}{Definitions}
\usepackage{paralist}
\usepackage{turnstile}
\usepackage{mdframed}
\usepackage{tikz}
\usepackage{caption}
\usepackage{newfloat}
\newtheorem{theorem}{Theorem}[section]
\newtheorem*{theorem*}{Theorem}

\newtheorem{proposition}[theorem]{Proposition}
\newtheorem*{proposition*}{Proposition}
\newtheorem{lemma}[theorem]{Lemma}
\newtheorem*{lemma*}{Lemma}
\newtheorem{corollary}[theorem]{Corollary}
\newtheorem*{corollary*}{Corollary}
\newtheorem*{conjecture*}{Conjecture}
\newtheorem{fact}[theorem]{Fact}
\newtheorem*{fact*}{Fact}

\newtheorem*{hypothesis*}{Hypothesis}

\theoremstyle{definition}
\newtheorem{definition}[theorem]{Definition}
\newtheorem*{definition*}{Definition}

\newtheorem{model}[theorem]{Model}

\theoremstyle{remark}

\newtheorem*{claim*}{Claim}

\newtheorem*{remark*}{Remark}

\newtheorem*{observation*}{Observation}

\newtheorem*{modell*}{Model}

\usepackage[
letterpaper,
top=1.2in,
bottom=1.2in,
left=1in,
right=1in]{geometry}
\usepackage{newpxtext} %
\usepackage{textcomp} %
\usepackage[varg,bigdelims]{newpxmath}
\usepackage[scr=rsfso]{mathalfa}%
\usepackage{bm} %
\let\mathbb\varmathbb
\usepackage{microtype}

\allowdisplaybreaks

\newcommand{\Set}[1]{\left\{#1\right\}}

\newcommand{\R}{{\mathbb R}}

\newcommand{\norm}[1]{\lVert #1 \rVert}

\newcommand{\Abs}[1]{\left\lvert #1 \right\rvert}

\let\epsilon=\varepsilon

\newcommand{\Paren}[1]{\left(#1\right)}

\newcommand{\Norm}[1]{\left\lVert#1\right\rVert}

\newcommand{\supp}{\operatorname{supp}}
\newcommand{\support}{\operatorname{supp}}

\newcommand{\Iprod}[1]{\left\langle#1\right\rangle}

\newcommand{\vectorize}{\operatorname{vec}}

\crefname{algocf}{Algorithm}{Algorithms}
\Crefname{algocf}{Algorithm}{Algorithms}
\usepackage{mdframed}
\usepackage[linesnumbered, ruled, vlined]{algorithm2e}

\begin{document}

\newcommand{\FormatAuthor}[3]{
\begin{tabular}{c}
#1 \\ {\small\texttt{#2}} \\ {\small #3}
\end{tabular}
}
\title{Lasso and Partially-Rotated Designs}

\author{}
\author{
\begin{tabular}[h!]{ccc}
  \FormatAuthor{Rares-Darius Buhai}{rares.buhai@inf.ethz.ch}{ETH Zürich}
\end{tabular}
} %
\date{\today}

\maketitle
\thispagestyle{empty}

\begin{abstract}
  We consider the sparse linear regression model $\bm{y} = X \beta +\bm{w}$, where $X \in \R^{n \times d}$ is the design, $\beta \in \R^{d}$ is a $k$-sparse secret, and $\bm{w} \sim N(0, I_n)$ is the noise.
Given input $X$ and $\bm{y}$, the goal is to estimate~$\beta$.
In this setting, the Lasso estimate achieves prediction error $O(k \log d / \gamma n)$, where $\gamma$ is the restricted eigenvalue (RE) constant of $X$ with respect to $\support(\beta)$.
In this paper, we introduce a new \emph{semirandom} family of designs --- which we call \emph{partially-rotated} designs --- for which the RE constant with respect to the secret is bounded away from zero even when a subset of the design columns are arbitrarily correlated among themselves.

As an example of such a design, suppose we start with some arbitrary $X$, and then apply a random rotation to the columns of $X$ indexed by $\support(\beta)$.
Let $\lambda_{\min}$ be the smallest eigenvalue of $\frac{1}{n} X_{\support(\beta)}^\top X_{\support(\beta)}$, where $X_{\support(\beta)}$ is the restriction of $X$ to the columns indexed by $\support(\beta)$.
In this setting, our results imply that Lasso achieves prediction error $O(k \log d / \lambda_{\min} n)$ with high probability.
This prediction error bound is independent of the arbitrary columns of $X$ not indexed by $\support(\beta)$, and is as good as if all of these columns were perfectly well-conditioned.

Technically, our proof reduces to showing that matrices with a certain deterministic property --- which we call \emph{restricted normalized orthogonality} (RNO) --- lead to RE constants that are independent of a subset of the matrix columns.
This property is similar but incomparable with the restricted orthogonality condition of~\cite{MR2243152-Candes05}.

\end{abstract}
\pagestyle{plain}
\setcounter{page}{1}

\section{Introduction}%
\label{sec:introduction}

Sparse linear models are among the simplest models that admit strong theoretical guarantees even when the number of samples is much smaller than the dimension.
Hence, these models are of fundamental interest in high-dimensional statistics.
In this paper we consider sparse linear regression,
where the input consists of a design $X \in \R^{n \times d}$ and a response vector $\bm{y} \in \R^{n}$, related linearly as
\begin{equation}
\label{eq:slr}
\bm{y} = X \beta + \bm{w}\,.
\end{equation}
Here $\beta \in \R^d$ is a $k$-sparse unknown \emph{secret} and $\bm{w} \sim N(0, \sigma^2 \cdot I_n)$ is noise.
Alternatively, one can view the input as consisting of $n$ samples $(X^i, \bm{y}_i)$ that satisfy the linear relation $\bm{y}_i = \langle X^i, \beta\rangle + \bm{w}_i$, where $X^i$ is the $i$-th row of the design.
We assume the standard normalization in which the columns of $X$ have norm $\sqrt{n}$.
Then, we focus on the goal of outputting an estimate $\hat{\beta} \in \R^d$ of $\beta$ such that the \emph{prediction error} $\frac{1}{n} \Norm{X\hat{\beta} - X\beta}_2^2$ is small.

In this setting, $\ell_0$-constrained optimization yields with high probability an estimate with prediction error $O\Paren{\sigma^2 k \log (d/k) / n}$~\cite{MR2882274-Raskutti11}.
This method is also known as the Best Subset Selection (BSS) algorithm.
Unfortunately, BSS runs in time exponential in $k$, and no polynomial-time algorithms are known that achieve \emph{any} non-trivial prediction error without additional assumptions.\footnote{A guarantee known as the \emph{slow rate} of lasso (see Theorem 7.20a in~\cite{MR3967104-Wainwright19}) is sometimes talked about as assumptionless, but we note that this guarantee scales linearly with the $\ell_1$-norm of $\beta$, so it is a non-trivial guarantee only with the additional assumption that this norm is small.}

\paragraph{Lasso}
Given this state of affairs, much research has been devoted to studying assumptions under which small prediction error can be achieved in polynomial time.
One of the most studied polynomial-time estimates for this problem is the Lasso~\cite{tibshirani1996regression}, defined (in its constrained form) as%
\begin{equation}
\label{eq:lasso}
\hat{\beta}_{\mathrm{Lasso}} = \min_{\substack{\hat{\beta} \in \R^d\\ \norm{\hat\beta}_1 \leq \norm{\beta}_1}} \Norm{\bm{y} - X \hat\beta}_2^2\,.  
\end{equation}
Lasso achieves with high probability prediction error 
\begin{equation}
\label{eq:one}
\frac{1}{n} \Norm{X\hat{\beta}_{\mathrm{Lasso}} - X\beta}_2^2 \leq O\Paren{\sigma^2 \frac{k \log d}{\gamma_{\supp(\beta)}(X) \cdot n}}\,,
\end{equation}
where $\gamma_{\supp(\beta)}(X)$ is the \emph{restricted eigenvalue} (RE) constant of $X$ with respect to $\supp(\beta)$ (see Theorem 7.20b in~\cite{MR3967104-Wainwright19}).
This guarantee is known as the \emph{fast rate} of Lasso.
Therefore, Lasso achieves small prediction error if $\gamma_{\supp(\beta)}(X)$ is bounded away from zero.
While other conditions have been studied under which Lasso succeeds, many of them imply in fact the RE condition~\cite{MR2576316-Geer09}.\footnote{An even weaker condition is the \emph{compatibility} condition introduced by~\cite{van2007deterministic}. This condition seems less well-known than the RE condition, so we restrict our attention to the latter.}

Thus, it is of interest to identify natural families of designs with RE constant bounded away from zero.
It is known that many \emph{random} designs satisfy this property: designs with i.i.d. {(sub-)Gaussian} or {sub-exponential} entries~\cite{MR2453368-Mendelson08, MR2796091-Adamczak11}, but also designs with rows sampled from some non-spherical {(sub-)Gaussian} distributions~\cite{JMLR:v11:raskutti10a, MR3061256-Rudelson13}, and designs with rows sampled from bounded orthonormal systems~\cite{MR2417886-Rudelson08}.
(Some of these satisfy in fact the stronger restricted isometry property (RIP), which implies RE constant bounded away from zero with respect to \emph{every} small set.)

In this paper, we introduce a new family of \emph{semirandom} designs --- which we call \emph{partially-rotated} designs --- that have RE constant bounded away from zero with respect to the secret even when a subset of the design columns are arbitrarily correlated among themselves.
We note that we adopt a weak definition of the RE constant, matching that in~\cite{MR2576316-Geer09}, but weaker than in some other sources~\cite{MR2533469-Bickel09, MR3025133-Negahban12, MR3967104-Wainwright19}.
See~\Cref{def:re} for our definition and~\Cref{sec:re} for a discussion about it.

\paragraph{Notation}
Let us introduce some notation necessary to state our results.
We say a vector is $k$-sparse if it has \emph{at most} $k$ non-zero entries.
For a vector $\beta$, we denote by $\supp(\beta)$ the support of $\beta$.
For a matrix $X \in \R^{n \times d}$ and a subset $S \subseteq [d]$, we denote by $X_S \in \R^{n \times |S|}$ the submatrix of $X$ restricted to the columns indexed by $S$.
For a subset $S \subseteq [d]$, we denote by $S^c$ the set $[d] \setminus S$.

\subsection{Main result: partially-rotated designs}

Let us start by defining the restricted eigenvalue constant of a matrix. Also see~\Cref{sec:re} for a discussion about our definition.

\begin{definition}[Restricted eigenvalue constant]
\label[definition]{def:re}
For a subset $S \subseteq [d]$, define the set $\mathbb{C}(S)$ as\footnote{Our results extend straightforwardly to a more general definition of $\mathbb{C}(S)$ with $\norm{\beta_{S^c}}_1 \leq L \norm{\beta_S}_1$. Definitions with $L>1$ are often used in the analysis of Lasso in its Lagrangian form.}
\[\mathbb{C}(S) = \Set{\beta \in \R^d \enspace:\enspace \norm{\beta_{S^c}}_1 \leq \norm{\beta_S}_1}\,.\]
Then, for a matrix $X \in \R^{n \times d}$, the \emph{restricted eigenvalue constant} of $X$ with respect to $S$ is 
\[\gamma_{S}(X) = \min_{\beta \in \mathbb{C}(S)}\frac{\frac{1}{n}\Norm{X\beta}_2^2}{\Norm{\beta_S}_2^2}\,.\]
\end{definition}

Next, we define the family of matrices that we introduce: partially-rotated matrices.
These are matrices in which a subset of the matrix columns are "randomly rotated" with respect to the other matrix columns.
Specifically, for some set $S \subseteq [d]$, we will require that an arbitrary linear combination of columns in $\bm{X}_S$ is uncorrelated with an arbitrary linear combination of columns in $\bm{X}_{S^c}$ with high probability.

\begin{definition}[Partially-rotated matrix]
\label[definition]{def:rot-sep}
A random matrix $\bm{X} \in \R^{n \times d}$ is $(\epsilon, \delta)$-partially-rotated with respect to a set $S \subseteq [d]$ if, for all $\alpha \in \R^{|S|}$ and $\beta \in \R^{|S^c|}$,
\[\mathbb{P}_{\bm{X}}\Paren{\Abs{\Iprod{\frac{\bm{X}_S \alpha}{\Norm{\bm{X}_S \alpha}_2}, \frac{\bm{X}_{S^c} \beta}{\Norm{\bm{X}_{S^c} \beta}_2}}} > \epsilon} \leq e^{-\delta n}\,.\]
\end{definition}

A natural way to generate a partially-rotated matrix is to start with some arbitrary matrix $X \in \R^{n \times d}$ and then apply a "random rotation" $\bm{R} \in \R^{n \times n}$ to the columns of $X$ in $S$ (or, equivalently, to the columns of $X$ in $S^c$).
We call the resulting matrix \emph{semirandom}, because the random rotation only serves to decorrelate $X_S$ from $X_{S^c}$, and the arbitrary structure within $X_S$ and $X_{S^c}$ is preserved.
In particular, the correlations of columns within $X_S$ and $X_{S^c}$ continue to be arbitrary.

We identify two natural families of random matrices $\bm{R} \in \R^{n \times n}$ whose application leads to partial rotation: random orthogonal matrices and random matrices with independent sub-Gaussian entries.
Let us state these implications formally (we defer the proofs to~\cref{sec:deferredproofs}).

\begin{lemma}[Random rotation implies partial rotation]
\torestate
{
\label{lem:rot-to-partial}
Let $\bm{R} \in \mathbb{R}^{n \times n}$ be a random matrix such that, for a fixed unit vector $v \in \R^d$, $\bm{R} v$ is distributed uniformly over the unit sphere.
(E.g., take $\bm{R}$ to be a random orthogonal matrix.)
Then, for any matrix $X \in \R^{n \times d}$ and any set $S \subseteq [d]$, the matrix $\bm{X}' = [X_S, \bm{R} X_{S^c}] \in \R^{n \times d}$ is $(0.01, c)$-partially rotated for some constant $c > 0$.
}
\end{lemma}

\begin{lemma}[Sub-Gaussianity implies partial rotation]
\torestate
{
\label{lem:subgauss-to-partial}
Let $\bm{R} \in \mathbb{R}^{n \times n}$ be a random matrix with independent entries that have mean zero, variance $\sigma^2$, and are $(\sigma')^2$-sub-Gaussian (see \cref{def:subgauss}).
Then, for any matrix $X \in \R^{n \times d}$ and any set $S \subseteq [d]$, the matrix $\bm{X}' = [X_S, \bm{R} X_{S^c}] \in \R^{n \times d}$ is $(0.01, c \cdot \sigma^4/(\sigma')^4)$-partially rotated for some constant $c > 0$.
}
\end{lemma}

We note in particular that \cref{lem:subgauss-to-partial} is satisfied with $\sigma,\sigma'=\Theta(1)$ by matrices $\bm{R}$ with i.i.d. standard Gaussian entries or with i.i.d. Rademacher entries.
Also see~\Cref{mod:semi} and~\Cref{mod:obl} for two applications in which partially-rotated matrices appear as designs.

Our main result is that, for a matrix $\bm{X}$ that is partially-rotated with respect to a set $S$, with high probability any subset $S' \subseteq S$ has RE constant $\gamma_{S'}(\bm{X})$ lower bounded by $\gamma_{S'}(\bm{X}_S)$ if $n$ is large enough. 
That is, the RE constant of $\bm{X}$ with respect to any subset $S' \subseteq S$ is independent of the columns of $\bm{X}$ in $S^c$.

\begin{theorem}[Main result]
\torestate
{
\label{thm:main}
Suppose a matrix $\bm{X} \in \R^{n \times d}$ with columns of norm $\sqrt{n}$ is $(0.01, c)$-partially-rotated with respect to a set $S \subseteq [d]$ for some constant $c > 0$.
Then there exists some constant $C > 0$ such that, with probability at least $1-e^{-\Omega(n)}$, for all $S' \subseteq S$ such that $n \geq C |S'| \log d / \gamma_{S'}(\bm{X}_S)$ it holds that
\[\gamma_{S'}(\bm{X}) \geq 0.9 \gamma_{S'}(\bm{X}_S)\,.\]
}
\end{theorem}
By the definition of the RE constant, we also have that $\gamma_{S'}(\bm{X}) \leq \gamma_{S'}({\bm{X}}_S)$, so the conclusion of the theorem holds in fact with approximate equality.

Let us state now explicitly the implication of~\Cref{thm:main} for the prediction error of Lasso.
Also see~\cref{eq:error-l1}, \cref{eq:error-l2}, and the surrounding discussion for implications for the recovery of $\beta$ itself.

\begin{corollary}[Main result for Lasso]
\torestate
{
\label{cor:main}
In the sparse linear regression model in~\Cref{eq:slr}, suppose that the design $\bm{X} \in \R^{n \times d}$ with columns of norm $\sqrt{n}$ is $(0.01, c)$-partially-rotated with respect to a set $S \subseteq [d]$ for some constant $c > 0$.
Then, if $\supp(\beta) \subseteq S$, the Lasso estimate $\hat{\beta}_{\mathrm{Lasso}} \in \R^d$ in~\Cref{eq:lasso} satisfies with high probability
\[\frac{1}{n} \Norm{\bm{X}\hat{\beta}_{\mathrm{Lasso}} - \bm{X}\beta}_2^2 \leq O\Paren{\sigma^2 \frac{k \log d}{\gamma_{\supp(\beta)}(\bm{X}_S) \cdot n}}\,.\]
}
\end{corollary}
The prediction error bound in~\Cref{cor:main} matches the \emph{fast rate} in \cref{eq:one}, but with restricted eigenvalue constant independent of the columns of $\bm{X}$ in $S^c$ --- or, in other words, as if the columns of $\bm{X}$ in $S^c$ were perfectly well-conditioned.
We remark that, to obtain this fast rate, it was important to prove~\Cref{thm:main} in the regime $n \gtrsim |S'| \log d / \gamma_{S'}(\bm{X}_S)$; a stronger assumption on $n$ would have implied a worse prediction error bound in~\Cref{cor:main}.

Finally, let us illustrate the applicability of the result through two specific models of partially-rotated designs.

\begin{model}[Semirandom i.i.d. Gaussian design]
\label[model]{mod:semi}
Consider an arbitrary matrix $X \in \R^{n \times d}$ with columns of norm $\sqrt{n}$ and an arbitrary $k$-sparse vector $\beta \in \R^{d}$.
Let the design $\bm{X}' \in \R^{n \times d}$ have columns $\bm{X}'_i \stackrel{i.i.d.}{\sim} N(0, I_n)$ (rescaled to have norm $\sqrt{n}$) if $i \in \supp(\beta)$ and $\bm{X}'_i = X_i$ otherwise.
Then $\bm{X}'$ is $(0.01, \Omega(1))$-partially-rotated with respect to $\supp(\beta)$.
\end{model}
\begin{proof}
The conclusion follows by \cref{lem:rot-to-partial} because of the rotation invariance of the Gaussian distribution.
\end{proof}

\Cref{mod:semi} is a relaxation of the well-studied model in which all entries of the design are i.i.d. standard Gaussian.
Standard results imply that with high probability ${\gamma_{\supp(\beta)}(\bm{X}_{\supp(\beta)}) \geq \Omega(1)}$ if $n \gtrsim k \log d$.
Therefore, when the secret is $\beta$, by~\Cref{cor:main} Lasso achieves prediction error $O(\sigma^2 k \log d / n)$.
This matches the error that would be achieved if all entries of the design were i.i.d. standard Gaussian, so to obtain this guarantee it suffices that the columns selected by the secret have planted i.i.d. standard Gaussian entries.

\begin{model}[Rotated adversary design]
\label[model]{mod:obl}
Consider an arbitrary matrix $X \in \R^{n \times d}$ with columns of norm $\sqrt{n}$.
Suppose an adversary selects $d'$ arbitrary vectors $a_1, \ldots, a_{d'} \in \R^{n}$ of norm $\sqrt{n}$.
Let $\bm{R} \in \R^{n \times n}$ be a random matrix corresponding to either \cref{lem:rot-to-partial} or \cref{lem:subgauss-to-partial}, and let the design $\bm{X}' \in \R^{n \times (d+d')}$ have columns $X_1, \ldots, X_d$, and $\bm{R} a_1, \ldots, \bm{R} a_{d'}$.
Then $\bm{X}'$ is partially-rotated with respect to $\{1, \ldots, d\}$.
\end{model}

To motivate~\Cref{mod:obl}, let us reinterpret the sparse linear regression problem as a denoising problem: given a sparse linear combination $\bm{y} = X \beta + \bm{w}$ of vectors $X_1, \ldots, X_d$ with added noise\linebreak $\bm{w} \sim N(0, \sigma^2 \cdot I_n)$, the goal is to estimate the sparse linear combination $X \beta$.
In this view, there is no need to think of the rows of $X$ as corresponding to "samples" --- the denoising problem makes sense on its own.
Then we can interpret~\Cref{mod:obl} as an oblivious adversary model, in which an adversary that does not know $X_1, \ldots, X_d$ is allowed to introduce new vectors in the input.
One way to model the adversary not knowing $X_1, \ldots, X_d$ is to say that the vectors it introduces should not match $X_1, \ldots, X_d$ more than up to a random rotation.
\Cref{cor:main} shows that for such an adversary Lasso has denoising guarantees that are independent of the correlations among the outlier vectors.

\subsection{Our framework: restricted normalized orthogonality}
\label{sec:intro-rno}

To prove~\Cref{thm:main}, we found it useful to introduce a new deterministic property of matrices that we call \emph{restricted normalized orthogonality} (RNO).
Then, we prove that partially-rotated matrices satisfy RNO with high probability, and that RNO implies the desired property for the RE constants of the matrix.

The RNO has a relatively simple definition and may be of independent interest, so we present it in some detail.
Informally, a matrix satisfies $(s,S,0.01)$-RNO if the $s$-sparse linear combinations of columns in $X_S$ have small correlations with the $s$-sparse linear combinations of those in $X_{S^c}$.

\begin{definition}[Restricted normalized orthogonality]
\label[definition]{def:rno}
Let $s \in \mathbb{N}$, $S \subseteq [d]$, and $\epsilon \geq 0$. 
A matrix $X \in \R^{n \times d}$ satisfies the $(s, S, \epsilon)$-\emph{restricted normalized orthogonality} (RNO) condition if, for all $s$-sparse $\alpha \in \R^{|S|}$ and $s$-sparse $\beta \in \R^{|S^c|}$,
\[\Abs{\Iprod{ \frac{X_S \alpha}{\Norm{X_S \alpha}_2}, \frac{X_{S^c} \beta}{\Norm{X_{S^c} \beta}_2} }} \leq \epsilon\,,\]
where we regard the inner product as $0$ when $X_S \alpha = 0$ or $X_{S^c} \beta = 0$.
\end{definition}

We then prove that, if $X$ satisfies $(\cdot , S, \cdot)$-RNO with appropriate parameters, then the RE constant of $X$ with respect to subsets $S' \subseteq S$ is independent of $X_{S^c}$.

\begin{proposition}[RNO implies RE]
\torestate
{
\label{thm:rno-to-re}
Let $s \in \mathbb{N}$, $S \subseteq [d]$, and $\epsilon \geq 0$. 
Suppose a matrix $X \in \R^{n \times d}$ with columns of norm $\sqrt{n}$ satisfies the $(s, S, \epsilon)$-RNO condition.
Then there exists some constant $C > 0$ such that, for any $S' \subseteq S$,
\[\gamma_{S'}(X) \geq \Paren{1 - \epsilon - C\sqrt{\frac{|S'|}{\gamma_{S'}(X_S) \cdot s}}} \gamma_{S'}(X_S)\,.\]
}
\end{proposition}

Hence, when $(s,S,\epsilon)$-RNO holds with $s \gtrsim |S'| / \gamma_{S'}(X_S)$ and with $\epsilon$ bounded away from one, then $\gamma_{S'}(X) \geq \Omega(\gamma_{S'}(X_S))$.
This is exactly the type of conclusion that we want in~\Cref{thm:main}.

Finally, we show that partially-rotated matrices satisfy RNO with high probability.

\begin{lemma}[Partially-rotated matrices satisfy RNO]
\torestate
{
\label{thm:rot-sep-rno}
Let $s \in \mathbb{N}$ and $S \subseteq [d]$.
Suppose a matrix $\bm{X} \in \R^{n \times d}$ is $(\epsilon, \delta)$-partially-rotated with respect to $S$ for some $\delta > 0$.
Then there exists some constant $C > 0$ such that, if $n \geq C \delta^{-1} s \max(\log d, \log(1/\epsilon))$, then $\bm{X}$ satisfies $(s, S, 2\epsilon)$-RNO with probability at least $1-e^{-\Omega(\delta n)}$.
}
\end{lemma}

\paragraph{RNO and other conditions}
Let us briefly comment on the relation between RNO and other deterministic conditions that appear in the context of Lasso.
First, we prove in~\Cref{sec:rip-to-rno} that RNO is implied by the restricted isometry property (RIP) of~\cite{MR2243152-Candes05}.
However, RIP is a much stronger property: it requires that \emph{all} sparse combinations of columns in $X$ are well-behaved, which may not be true with RNO.

Second, we note that RNO is similar to the restricted orthogonality condition of~\cite{MR2243152-Candes05}, the main difference being that the latter considers inner products $\Abs{\Iprod{ \frac{X_S \alpha}{\Norm{\alpha}_2}, \frac{X_{S^c} \beta}{\Norm{\beta}_2} }}$.
However, this difference in normalization makes the latter depend on the correlations among columns of $X_{S^c}$: the magnitude of $X_{S^c} \beta / \Norm{\beta}_2$ can be larger by a factor of $\sqrt{s}$ when $X_{S^c}$ has identical columns than when it has random columns.
For our application to partially-rotated matrices we need to tolerate arbitrary correlations among the columns of $X_{S^c}$, and our guarantees would be weaker with an additional factor of $\sqrt{s}$.
The normalization in our definition avoids these issues.

\subsection{Discussion: restricted eigenvalue constant}
\label{sec:re}
Our definition of the RE constant in~\Cref{def:re} matches that of~\cite{MR2576316-Geer09}.
It is also common, however, to define the RE constant with $\norm{\beta}_2^2$ in the denominator instead of $\norm{\beta_S}_2^2$~\cite{MR2533469-Bickel09, MR3025133-Negahban12, MR3967104-Wainwright19}.
Specifically, the alternative version, which we denote by $\gamma'_S(X)$, is
\begin{equation}
\label{eq:alt-re}
\gamma'_{S}(X) = \min_{\beta \in \mathbb{C}(S)}\frac{\frac{1}{n}\Norm{X\beta}_2^2}{\Norm{\beta}_2^2}\,.
\end{equation}

Because $\norm{\beta}_2^2 \geq \norm{\beta_S}_2^2$, we have $\gamma'_S(X) \leq \gamma_S(X)$ for all $X$ and $S$.
It is in fact possible that $\gamma_S(X)$ is bounded away from zero but $\gamma'_S(X)$ is not.
It turns out that our guarantees in~\Cref{thm:main} and~\Cref{thm:rno-to-re} do not hold for $\gamma'_S(X)$: see~\Cref{sec:strong-re} for a counterexample.

Many of the known guarantees of Lasso hold with our definition~\cite{MR2576316-Geer09}: the prediction error bound in~\Cref{eq:one}, and the parameter error bounds
\begin{equation}
\label{eq:error-l1}
\norm{\hat\beta-\beta}_1 \leq O(\sigma k \sqrt{\log d} / \gamma_{\supp(\beta)}(X) \sqrt{n})
\end{equation}
and 
\begin{equation}
\label{eq:error-l2}
\norm{\hat\beta_{\supp(\beta)} - \beta}_2^2 \leq O(\sigma^2 k \log d / (\gamma_{\supp(\beta)}(X))^2 n)\,.
\end{equation}
However, the parameter error guarantee~\cite{MR3025133-Negahban12, MR3967104-Wainwright19}
\[
\norm{\hat\beta - \beta}_2^2 \leq O(\sigma^2 k \log d / (\gamma'_{\supp(\beta)}(X))^2 n)
\]
does not hold with $(\gamma_{\supp(\beta)}(X))^2$, so in our setting we do not achieve this last guarantee.

\subsection{Related work}

Other works that have studied sparse linear regression in settings with correlated features include~\cite{elasticnet05, huang2011variable, BUHLMANN20131835, NIPS20136e2713a6, pmlr-v51-figueiredo16, 8646615, kelner2022power, kelner2024feature, kelner2024lasso}.
In their settings the features selected by the secret may be correlated among themselves, and therefore the RE constant may not be bounded away from zero.
Then, these papers study Lasso \emph{after} a preprocessing step, or else algorithms different from Lasso.

Another related line of work has studied robust versions of sparse linear regression~\cite{karmalkar2018compressed, dalalyan2019outlier, liu2020high, d2021consistent2, kelner2023semi}.
In these papers, an adversary is allowed to corrupt a fraction of the samples in the input, and the question is whether polynomial-time algorithms can still recover a good estimate.
The closest in spirit to our setting is~\cite{kelner2023semi}, which considers a semirandom design in which only a subset of the rows is well-conditioned.
Nevertheless, these results are incomparable with ours, because in our setting we have ill-behaved design \emph{columns}, not rows (samples).

\subsection{Organization}
In~\Cref{sec:tech} we give a technical overview of our proofs.
In~\Cref{sec:sparse} we prove a sparsification result used in the proof of~\Cref{thm:rno-to-re}, and then in \Cref{sec:rno-to-re} we prove~\Cref{thm:rno-to-re} and in~\Cref{sec:rot-sep-rno} we prove~\Cref{thm:rot-sep-rno}.
Finally, \Cref{sec:main-results} proves \Cref{thm:main}.

\section{Techniques}
\label{sec:tech}

In this section we give an overview of our proofs. 
For simplicity, we focus on the special case where a matrix is partially-rotated with respect to a set of cardinality $k$.
Specifically, let $\bm{X} \in \R^{n \times d}$ have columns of norm $\sqrt{n}$ and be partially-rotated with respect to a set $S \subseteq [d]$ with $|S|=k$.

Suppose we were trying to prove~\Cref{thm:main} directly.
The critical fact to prove turns out to be that, for $n$ large enough, with high probability $\Norm{\bm{X}_S \alpha + \bm{X}_{S^c} \beta}^2 \geq 0.99 \Norm{\bm{X}_S \alpha}^2$ for all $\alpha \in \R^{k}$ and $\beta \in \R^{|S^c|}$ with $\norm{\beta}_1 \leq \norm{\alpha}_1$.
That is, linear combinations of columns in $\bm{X}_S$ cannot be canceled by appropriate linear combinations of columns in $\bm{X}_{S^c}$.
Quantitatively:

\begin{fact}
\label[fact]{fact:ftech}
Suppose $\Norm{\bm{X}_S \alpha}_2 \geq \sqrt{\gamma n} \Norm{\alpha}_2$ for all $\alpha \in \R^k$.
If $n \geq C k \log d / \gamma$ for some absolute constant $C > 0$, then with high probability, for all $\alpha \in \R^{k}$ and $\beta \in \R^{|S^c|}$ with $\norm{\beta}_1 \leq \norm{\alpha}_1$,
\begin{equation}
\label{eq:fact}
\Norm{\bm{X}_S \alpha + \bm{X}_{S^c} \beta}_2^2 \geq 0.99 \Norm{\bm{X}_S \alpha}_2^2\,.
\end{equation}
\end{fact}

We note that, to prove~\Cref{fact:ftech}, it would suffice to prove that for all $\alpha$ and $\beta$
\begin{equation}
\label{eq:ip}
\Abs{\Iprod{\frac{\bm{X}_S \alpha}{\Norm{\bm{X}_S\alpha}_2}, \frac{\bm{X}_{S^c}\beta}{\Norm{\bm{X}_{S^c}\beta}_2}}} \leq 0.01\,.
\end{equation}
However, this inequality does not hold: in particular, if the columns of $\bm{X}_{S^c}$ span all of $\R^n$, then it is always possible to choose $\beta$ such that the inner product is equal to $1$.

\paragraph{Attempt via \Holder's inequality}
Because we have constraints in terms of the $\ell_1$-norms of $\alpha$ and $\beta$, a natural approach is to apply \Holder's inequality as
\[\Abs{\Iprod{\frac{\bm{X}_S \alpha}{\Norm{\bm{X}_S\alpha}_2}, \frac{\bm{X}_{S^c}\beta}{\Norm{\bm{X}_{S^c}\beta}_2}}} = \Abs{\Iprod{\frac{\bm{X}_{S^c}^\top \bm{X}_S \alpha}{\Norm{\bm{X}_S \alpha}_2}, \frac{\beta}{\Norm{\bm{X}_{S^c}\beta}_2}}} \leq \frac{\Norm{\bm{X}_{S^c}^\top \bm{X}_S \alpha}_\infty}{\Norm{\bm{X}_S \alpha}_2} \cdot \frac{\Norm{\beta}_1}{\Norm{\bm{X}_{S^c} \beta}_2}\,.\]
Using the definition of partial rotation, we can prove by standard concentration bounds that with high probability $\Norm{\bm{X}_{S^c}^\top \bm{X}_S \alpha}_\infty \leq O(\sqrt{k \log d}) \norm{\bm{X}_S \alpha}_2$ for all $\alpha \in \R^{k}$.
For the other terms, we have $\norm{\beta}_1 \leq \norm{\alpha}_1 \leq \sqrt{k} \norm{\alpha}_2$ and, because~\Cref{fact:ftech} is trivially true when $\Norm{\bm{X}_{S^c}\beta}_2 \ll \Norm{\bm{X}_S \alpha}_2$, also $\Norm{\bm{X}_{S^c}\beta}_2 \geq \Omega\Paren{\Norm{\bm{X}_S\alpha}_2} \geq \Omega(\sqrt{\gamma n} \norm{\alpha}_2)$.
Combining these, we get with high probability
\[\Abs{\Iprod{\frac{\bm{X}_S \alpha}{\Norm{\bm{X}_S\alpha}_2}, \frac{\bm{X}_{S^c}\beta}{\Norm{\bm{X}_{S^c}\beta}_2}}} \leq O\Paren{\sqrt{\frac{k^2 \log d}{\gamma n}}} \,,\]
which is $\leq 0.01$ when $n \geq C k^2 \log d / \gamma$.
Recall that we aim for a linear dependence on $k$, so we need to do better.

\paragraph{Result for sparse linear combinations}
An important observation is that, if $n \geq C s \log d$, then~\Cref{eq:ip} is true with high probability for all $s$-sparse $\alpha$ and $\beta$.
In other words, $\bm{X}$ satisfies $(s, S, 0.01)$-RNO with high probability.
This statement is our~\Cref{thm:rot-sep-rno}.
The proof involves an $\epsilon$-net argument combined with the definition of partial rotation.

Hence, because we are interested in the case $n \geq C k \log d / \gamma$, the conclusion of~\Cref{fact:ftech} is true for all $O(k/\gamma)$-sparse $\alpha$ and $\beta$.
Recall that in this section we consider the special case $|S|=k$, so all $\alpha$ are $k$-sparse and hence also $O(k/\gamma)$-sparse.
However, $\beta \in \R^{|S^c|}$ may not be sparse.
How can we then use this result?

\paragraph{Sparsification}
Suppose we had an $O(k/\gamma)$-sparse approximation $\beta'$ of $\beta$.
Then, using the triangle inequality, we could write
\begin{equation}
\label{eq:sparse-eq}
\Norm{\bm{X}_S \alpha + \bm{X}_{S^c}\beta}_2 \geq \Norm{\bm{X}_S \alpha + \bm{X}_{S^c}\beta'}_2 - \Norm{\bm{X}_{S^c}\beta' - \bm{X}_{S^c} \beta}_2\,.
\end{equation}
Because $\beta'$ is $O(k/\gamma)$-sparse, we can apply the conclusion of~\Cref{fact:ftech} to the first term on the right-hand side and lower bound it by $\Omega(\Norm{\bm{X}_S \alpha}_2)$.
But how to bound the term $\Norm{\bm{X}_{S^c} \beta' - \bm{X}_{S^c}\beta}_2$?
In other words, how to control the quality of the approximation?

It turns out that an effective approximation follows from an approximate Carath\'eodory's theorem obtained by subsampling the coordinates of $\beta$ (see Theorem 0.0.2 in~\cite{Vershynin2018}).
Given the approximate Carath\'eodory's theorem, we easily derive the following approximation lemma (similar sparsification results have been observed before in the statistics literature; e.g., see Lemma 5.1 in~\cite{MR3568047-Oliveira16} and Lemma 2.7 in~\cite{MR3612870-Lecue17}):

\begin{lemma}[Sparsification with $\ell_1$-norm error]
\torestate
{
\label{cor:caratheodory}
Let $X \in \R^{n \times d}$. For each $\beta \in \R^d$ and $1 \leq s \leq d$, there exists some $s$-sparse $\beta' \in \R^d$ such that 
\[\Norm{X\beta - X\beta'}_2 \leq \frac{2 \max\{\Norm{X_1}_2, \ldots, \Norm{X_d}_2\}}{\sqrt{s}} \norm{\beta}_1\,.\]
}
\end{lemma}

Hence, by choosing $s=O(k/\gamma)$ large enough in~\Cref{cor:caratheodory}, there exists some $O(k/\gamma)$-sparse $\beta'$ such that $\Norm{\bm{X}_{S^c} \beta' - \bm{X}_{S^c} \beta}_2 \leq c \sqrt{\gamma n/k} \norm{\beta}_1$ for some small enough $c > 0$.
Plugging this bound back into~\Cref{eq:sparse-eq}, and using that $\norm{\beta}_1 \leq \norm{\alpha}_1 \leq \sqrt{k} \norm{\alpha}_2$ and $\norm{\alpha}_2 \leq \frac{1}{\sqrt{\gamma n}} \norm{\bm{X}_S \alpha}_2$, gives the desired result.
This part of the argument, properly generalized, is our~\Cref{thm:rno-to-re}.
\section{Sparsification with $\ell_1$-norm error}
\label{sec:sparse}

We first state Theorem 0.0.2 of~\cite{Vershynin2018}, which gives an approximate Carath\'eodory's theorem.
The proof of this result is by subsampling: it shows that, given a set $T$ and a point $x \in \operatorname{conv}(T)$, if we sample $s$ points from $T$ with probability proportional to their weight in the convex combination $x$, then the expected distance of $x$ from their average is bounded by $\operatorname{diam}(T)/\sqrt{s}$.
This argument is often referred to as Maurey's empirical method~\cite{pisier1981remarques}.

\begin{fact}[See Theorem 0.0.2 in~\cite{Vershynin2018}]
\label[fact]{fact:caratheodory}
Consider a set $T \subset \mathbb{R}^n$ whose diameter is bounded by $1$.
Then, for every point $x \in \operatorname{conv}(T)$ and every integer $s$, one can find points $x_1, \ldots, x_s \in T$ such that
\[\Norm{x - \frac{1}{s} \sum_{i=1}^s x_i}_2 \leq \frac{1}{\sqrt{s}}\,.\]
\end{fact}

Based on this result, we derive a sparsification result for our setting with error depending on the vector's $\ell_1$-norm.
We note that similar sparsification results have been observed before in the statistics literature; e.g., see Lemma 5.1 in~\cite{MR3568047-Oliveira16} and Lemma 2.7 in~\cite{MR3612870-Lecue17}.

\restatelemma{cor:caratheodory}
\begin{proof}
Let $D = \max\{\Norm{X_1}_2, \ldots, \Norm{X_d}_2\}$ be the maximum norm of a column of $X$.
Then let $T = \{X_1 / 2D, \ldots, X_d / 2D, -X_1 / 2D, \ldots, -X_d / 2D\}$ be the set of columns of $X$ and their negations divided by $2D$.
The diameter of $T$ is then bounded by $1$.

Note that 
\[\frac{X\beta}{2D \norm{\beta}_1} = \sum_{i=1}^d \frac{\beta_i}{2D \sum_{j=1}^d |\beta_j|} X_i = \sum_{i=1}^d \frac{|\beta_i|}{\sum_{j=1}^d |\beta_j|} (\operatorname{sign}(\beta_i) X_i / 2D)\]
is a convex combination of points in $T$. Then, by~\Cref{fact:caratheodory}, there exist $x_1, \ldots, x_s \in T$ such that $\frac{X\beta}{2D\norm{\beta}_1}$ has distance at most $1/\sqrt{s}$ from $\frac{1}{s} \sum_{i=1}^s x_i$. 
Observe that $\frac{1}{s} \sum_{i=1}^s x_i$ can be written as $X \beta'$ for some $s$-sparse $\beta' \in \R^d$.
Thus
\[\Norm{\frac{X\beta}{2D \norm{\beta}_1} - X\beta'}_2 \leq \frac{1}{\sqrt{s}}\,.\]
Multiplying both sides by $2D \norm{\beta}_1$ gives the desired conclusion.
\end{proof}

\section{Proof that RNO implies RE}
\label{sec:rno-to-re}

\restateprop{thm:rno-to-re}
\begin{proof}
Consider some $z \in \R^d$ with $z \in \mathbb{C}(S')$, and let $\alpha \in \R^{|S|}$ and $\beta \in \R^{|S^c|}$ such that $z_S = \alpha$ and $z_{S^c} = \beta$.
We have $Xz = X_S \alpha + X_{S^c} \beta$.
Then, by~\Cref{cor:caratheodory}, there exist $s$-sparse $\alpha'$ and $\beta'$ with $\Norm{X_S \alpha' - X_S \alpha}_2 \leq 2\sqrt{\frac{n}{s}}\norm{\alpha}_1$ and $\Norm{X_{S^c} \beta' - X_{S^c} \beta}_2 \leq 2\sqrt{\frac{n}{s}}\norm{\beta}_1$.
Hence
\[\Norm{Xz}_2 \geq \Norm{X_S\alpha' + X_{S^c}\beta'}_2 - O\Paren{\sqrt{\frac{n}{s}}}(\Norm{\alpha}_1+\Norm{\beta}_1)\,.\]
If $X_S \alpha' = 0$ or $X_{S^c} \beta' = 0$, the next step will hold trivially, so assume that both are non-zero.
Then, using that $X$ satisfies $(s,S, \delta)$-RNO and that $\alpha'$ and $\beta'$ are $s$-sparse, we get that
\begin{align*}
\Norm{X_S\alpha' + X_{S^c}\beta'}_2^2
&\geq \Norm{X_s \alpha'}_2^2 + \Norm{X_{S^c}\beta'}_2^2 - 2 \Norm{X_S \alpha'}_2 \Norm{X_{S^c} \beta'}_2 \Abs{\Iprod{ \frac{X_S \alpha'}{\Norm{X_S \alpha'}_2}, \frac{X_{S^c} \beta'}{\Norm{X_{S^c} \beta'}_2} }}\\
&\geq \Norm{X_s \alpha'}_2^2 + \Norm{X_{S^c}\beta'}_2^2 - 2 \delta \Norm{X_S \alpha'}_2 \Norm{X_{S^c} \beta'}_2\\
&= (1-\delta) \Norm{X_s \alpha'}_2^2 + (1-\delta) \Norm{X_{S^c}\beta'}_2^2 + \Paren{\sqrt{\delta}\Norm{X_s \alpha'}_2 - \sqrt{\delta}\Norm{X_{S^c}\beta'}_2}^2\\
&\geq (1-\delta) \Norm{X_s \alpha'}_2^2\,.
\end{align*}
Then, replacing $\Norm{X_S\alpha' + X_{S^c}\beta'}_2$ by $(1-\delta) \Norm{X_s \alpha'}_2$ in our lower bound on $\Norm{Xz}_2$, and using again that $\Norm{X_S \alpha' - X_S \alpha}_2 \leq \sqrt{\frac{2n}{s}}\norm{\alpha}_1$, we get that
\[\Norm{Xz}_2 \geq \sqrt{1-\delta} \Norm{X_S\alpha}_2 - O\Paren{\sqrt{\frac{n}{s}}}(\Norm{\alpha}_1+\Norm{\beta}_1)\,.\]
Let us use now that $z \in \mathbb{C}(S')$. This implies that $\norm{z_{(S')^c}}_1 \leq \norm{z_{S'}}_1$, and therefore, because $S' \subseteq S$, also that $\norm{\beta}_1 \leq \norm{\alpha_{S'}}_1 \leq \norm{\alpha}_1$. Hence
\[\Norm{Xz}_2 \geq \sqrt{1-\delta} \Norm{X_S\alpha}_2 - O\Paren{\sqrt{\frac{n}{s}}}\Norm{\alpha}_1\,.\]
The fact that $\norm{z_{(S')^c}}_1 \leq \norm{z_{S'}}_1$ also implies that $\norm{\alpha_{(S')^c}}_1 \leq \norm{\alpha_{S'}}_1$.
In particular, this implies that $\alpha \in \mathbb{C}(S')$, so $\Norm{X_S\alpha}_2 \geq \sqrt{n \gamma_{S'}(X_S)} \norm{\alpha_{S'}}_2$.
Finally, $\alpha \in \mathbb{C}(S')$ also implies that $\norm{\alpha}_1 \leq 2 \norm{\alpha_{S'}}_1 \leq 2 \sqrt{|S'|} \norm{\alpha_{S'}}_2$, so putting everything together
\begin{align*}
\Norm{Xz}_2
\geq \Paren{\sqrt{(1-\delta)n \gamma_{S'}(X_S)} - O\Paren{\sqrt{\frac{|S'| n}{s}}}} \norm{\alpha_{S'}}_2\,,
\end{align*}
and using that $\alpha_{S'} = z_{S'}$, squaring, and rearranging,
\begin{align*}
\frac{\frac{1}{n}\Norm{Xz}_2^2}{\norm{z_{S'}}_2^2}
&\geq \Paren{\sqrt{1-\delta} - O\Paren{\sqrt{\frac{|S'|}{\gamma_{S'}(X_S) \cdot s}}}}^2 \gamma_{S'}(X_S)\\
&\geq \Paren{1 - \delta - 2\sqrt{1-\delta} \cdot O\Paren{\sqrt{\frac{|S'|}{\gamma_{S'}(X_S) \cdot s}}}} \gamma_{S'}(X_S)\\
&\geq \Paren{1 - \delta - C \sqrt{\frac{|S'|}{\gamma_{S'}(X_S) \cdot s}}} \gamma_{S'}(X_S)
\end{align*}
for some large enough constant $C$.
Because this holds for all $z \in \mathbb{C}(S')$, we obtain the desired lower bound on $\gamma_{S'}(X)$.
\end{proof}

\section{Proof that partially-rotated matrices satisfy RNO}
\label{sec:rot-sep-rno}

\restatelemma{thm:rot-sep-rno}
\begin{proof}
Consider some fixed sets $S_\alpha \subseteq S$ with $|S_\alpha| = s$ and $S_\beta \subseteq S$ with $|S_\beta| = s$.
We first prove the result for $\alpha$ and $\beta$ satisfying $\supp(\alpha) \subseteq S_\alpha$ and $\supp(\beta) \subseteq S_\beta$, and then union bound over all $\leq d^{2s}$ choices of $S_\alpha$ and $S_\beta$.
Let $\bm{\Phi}_{\alpha} \in \R^{n \times s}, \bm{\Phi}_{\beta} \in \R^{n \times s}$ be orthonormal bases of the span of the columns of $\bm{X}_S$ indexed by $S_\alpha$ and the span of the columns of $\bm{X}_{S^c}$ indexed by $S_\beta$, respectively.
Then 
\begin{align*}
\Abs{\Iprod{ \frac{\bm{X}_S \alpha}{\Norm{\bm{X}_S \alpha}_2}, \frac{\bm{X}_{S^c} \beta}{\Norm{\bm{X}_{S^c} \beta}_2} }} = \Abs{\Iprod{ \bm{\Phi}_{\alpha} v_\alpha, \bm{\Phi}_{\beta} v_\beta}}\,,
\end{align*}
where $v_\alpha \in \R^s$ and $v_\beta \in \R^s$ are such that $\bm{\Phi}_\alpha v_\alpha = \frac{\bm{X}_S \alpha}{\Norm{\bm{X}_S \alpha}_2}$ and $\bm{\Phi}_\beta v_\beta = \frac{\bm{X}_{S^c} \beta}{\Norm{\bm{X}_{S^c} \beta}_2}$.
By $(\epsilon, \delta)$-partial rotation, we have for each $v_\alpha$ and $v_\beta$ that $\mathbb{P}(\Abs{\Iprod{ \bm{\Phi}_{\alpha} v_\alpha, \bm{\Phi}_{\beta} v_\beta}} > \epsilon) \leq e^{-\delta n}$.

Consider an $\epsilon/100$-cover $\{v_1, \ldots, v_m\}$ of the unit $s$-dimensional ball. Such a cover exists with $m \leq O(1/\epsilon)^{s}$ (see Example 5.8 in~\cite{MR3967104-Wainwright19}).
By a union bound over all $m^2$ pairs of vectors in the net, for $n \geq C \delta^{-1} s \log(1/\epsilon)$ large enough, we also get with probability at least $1-e^{-\Omega(\delta n)}$ that $\Abs{\Iprod{ \bm{\Phi}_{\alpha} v_i, \bm{\Phi}_{\beta} v_j}} \leq \epsilon$ for all $i, j \in [m]$.
Then, for some arbitrary $v_\alpha$ and $v_\beta$, we can select $v'_\alpha$ and $v'_\beta$ in the cover with $\Norm{v'_\alpha - v_\alpha}_2 \leq \epsilon/100$ and $\Norm{v'_\beta - v_\beta}_2 \leq \epsilon/100$, so by the triangle inequality
\begin{align*}
\Abs{\Iprod{ \bm{\Phi}_{\alpha} v_\alpha, \bm{\Phi}_{\beta} v_\beta}}
&\leq \Abs{\Iprod{ \bm{\Phi}_{\alpha} v'_\alpha, \bm{\Phi}_{\beta} v'_\beta}} + \Abs{\Iprod{ \bm{\Phi}_{\alpha} v_\alpha, \bm{\Phi}_{\beta} (v'_\beta - v_\beta)}} + \Abs{\Iprod{ \bm{\Phi}_{\alpha} (v'_\alpha - v_\alpha), \bm{\Phi}_{\beta} v'_\beta}}\\
&\leq \epsilon + \epsilon/100 + \epsilon/100 \leq 2\epsilon\,.
\end{align*}

Finally, by union bounding over all $\leq d^{2s}$ choices of $S_\alpha$ and $S_\beta$, we get the same bound with probability at least $1-e^{-\Omega(\delta n)}$ for all $s$-sparse $\alpha$ and $s$-sparse $\beta$, for $n \geq C \delta^{-1} s \log d$ large enough.
\end{proof}

\section{Proof of main result}
\label{sec:main-results}

Here we finally prove~\Cref{thm:main}.
The proof of \Cref{cor:main} is deferred to \cref{sec:deferredproofs}.

\restatetheorem{thm:main}
\begin{proof}(Proof of~\Cref{thm:main})
Consider some fixed $S' \subseteq S$ such that $n \geq C |S'| \log d / \gamma_{S'}(\bm{X}_S)$ for some large enough constant $C > 0$.
Then we get by~\Cref{thm:rot-sep-rno} that $\bm{X}$ satisfies $(C' |S'| / \gamma_{S'}(\bm{X}_S), S, 0.02)$-RNO for some large enough constant $C' > 0$ with probability at least $1-e^{-c n}$ for some small enough constant $c > 0$.
Then, by~\Cref{thm:rno-to-re}, $\gamma_{S'}(\bm{X}) \geq 0.9 \gamma_{S'}(\bm{X}_S)$.

It remains to union bound over all choices $S' \subseteq S$ with $n \geq C |S'| \log d / \gamma_{S'}(\bm{X}_S)$.
Note that, under this lower bound on $n$, we have $|S'| \leq \frac{n \cdot \gamma_{S'}(\bm{X}_S)}{C \log d} \leq \frac{n}{C \log d}$, where we used that always $\gamma_{S'}(\bm{X}_S) \leq 1$.
It suffices to union bound over only the maximal $S'$.
Then, by union bounding over all $\binom{d}{|S'|} \leq d^{|S'|} \leq d^{n/(C \log d)} = e^{n/C}$ choices, we get probability of success $1 - e^{-cn + n/C}$, which for $C > 1/c$ is $1-e^{-\Omega(n)}$.
\end{proof}

\bibliographystyle{alpha}
\bibliography{references}

\appendix 
\section{Concentration bounds}
\label{sec:conc}
In this section we state some standard concentration bounds that we use in our proofs.

\begin{definition}[Sub-Gaussian random variable]
\label[definition]{def:subgauss}
A random variable $\bm{x} \in \R$ is $\sigma^2$-sub-Gaussian if $\mathbb{E} e^{\lambda (\bm{x} - \mathbb{E} \bm{x})} \leq e^{\lambda^2 \sigma^2 / 2}$ for all $\lambda \in \R$.
\end{definition}

\begin{fact}[See text before Proposition 2.5 in~\cite{MR3967104-Wainwright19}]
\label[fact]{fact:sum-subgauss}
If two random variables $\bm{x}_1, \bm{x}_2 \in \R$ are independent and $\sigma_1^2$-sub-Gaussian and $\sigma_2^2$-sub-Gaussian, respectively, then $\bm{x}_1 + \bm{x}_2$ is $(\sigma_1^2+\sigma_2^2)$-sub-Gaussian.
\end{fact}

\begin{fact}[See Equation 2.9 in~\cite{MR3967104-Wainwright19}]
\label[fact]{fact:uni-subgauss}
Let $\bm{x} \in \R$ be $\sigma^2$-sub-Gaussian. Then 
\[\mathbb{P}\Paren{\Abs{\bm{g}} \geq t} \leq 2e^{-t^2/(2\sigma^2)}\,.\]
\end{fact}

\begin{fact}[See Example 2.11 in~\cite{MR3967104-Wainwright19}]
\label[fact]{fact:gauss-norm}
Let $\bm{g} \sim N(0, I_n)$. Then for $t \in (0, 1)$
\[\mathbb{P}\Paren{\Abs{\frac{1}{n} \Norm{\bm{g}}_2^2 - 1} \geq t} \leq 2\exp\Paren{-\frac{t^2 n}{8}}\,.\]
\end{fact}

\begin{fact}[Hanson-Wright, see Theorem 1.1 in~\cite{MR3125258-Rudelson13}]
\label{fact:hanson-wright}
Let $X \in \R^n$ have independent $\sigma^2$-sub-Gaussian entries with mean $0$.
Then there exists some absolute constant $c > 0$ such that, for all $t \geq  0$,
\[\mathbb{P}\Paren{|\bm{X}^\top A \bm{X} - \mathbb{E} \bm{X}^\top A \bm{X}| > t} \leq 2\exp\Paren{- c \cdot \min\Paren{\frac{t^2}{\sigma^4 \Norm{A}_F^2}, \frac{t}{\sigma^2 \Norm{A}}}}\,.\]
\end{fact}

\section{Deferred proofs}
\label{sec:deferredproofs}

\restatelemma{lem:rot-to-partial}
\begin{proof}
We have 
\begin{align*}
\Iprod{\frac{\bm{X}'_S \alpha}{\Norm{\bm{X}'_S \alpha}_2}, \frac{\bm{X}'_{S^c} \beta}{\Norm{\bm{X}'_{S^c} \beta}_2}}
&= \Iprod{\frac{X_S \alpha}{\Norm{X_S \alpha}_2}, \frac{\bm{R} X_{S^c} \beta}{\Norm{\bm{R} X_{S^c} \beta}_2}}
= \Iprod{\frac{X_S \alpha}{\Norm{X_S \alpha}_2}, \bm{R} \frac{X_{S^c} \beta}{\Norm{X_{S^c} \beta}_2}}\,,
\end{align*}
where in the last equality we used that $\ell_2$-norms are rotation-invariant.
So, equivalently, it suffices to bound $\mathbb{P} \Paren{\Abs{\langle v, \bm{R} w\rangle} > \epsilon}$ for arbitrary fixed unit vectors $v, w \in \R^n$.

We have that $\bm{R} w$ is distributed uniformly over the unit sphere.
Therefore $\bm{R} w$ is distributed as $\bm{g}/\norm{\bm{g}}$, where $\bm{g} \sim N(0, I_n)$.
Then $\langle v, \bm{R} w\rangle$ is distributed as $\langle v, \bm{g}\rangle / \norm{\bm{g}}$.
Because $\langle v, \bm{g}\rangle \sim N(0, 1)$, we have by~\Cref{fact:uni-subgauss} that $\mathbb{P}(|\langle v, \bm{g}\rangle| \geq \epsilon \sqrt{n} / 2) \leq 2 e^{-\epsilon^2 n/8}$.
In addition, we have by~\Cref{fact:gauss-norm} that $\mathbb{P}(\norm{\bm{g}}_2 < \sqrt{n} / 2) \leq 2e^{-n/32}$.
Therefore, 
\[\mathbb{P} \Paren{\Abs{\langle v, \bm{R} w\rangle} > \epsilon} \leq 2 e^{-\epsilon^2 n/8} + 2 e^{-n/32} \leq e^{-\delta n}\]
for $\delta = \Omega(\epsilon^2)$ small enough, and then $\bm{X}'$ is $(\epsilon, \delta)$-partially rotated.
Choosing $\epsilon=0.01$ gives the stated result.
\end{proof}

\restatelemma{lem:subgauss-to-partial}
\begin{proof}
We have 
\begin{align*}
\Iprod{\frac{\bm{X}'_S \alpha}{\Norm{\bm{X}'_S \alpha}_2}, \frac{\bm{X}'_{S^c} \beta}{\Norm{\bm{X}'_{S^c} \beta}_2}}
&= \Iprod{\frac{X_S \alpha}{\Norm{X_S \alpha}_2}, \frac{\bm{R} X_{S^c} \beta}{\Norm{\bm{R} X_{S^c} \beta}_2}}\,,
\end{align*}
so, equivalently, it suffices to bound $\mathbb{P} \Paren{\Abs{\langle v, \frac{\bm{R} w}{\Norm{\bm{R} w}_2}\rangle} > \epsilon}$ for arbitrary fixed unit vectors $v, w \in \R^n$.

First, we show that $\Abs{\langle v, \bm{R} w\rangle}$ is small.
We have that $\langle v, \bm{R} w \rangle = \langle \bm{R}, v \otimes w\rangle$, which by~\cref{fact:sum-subgauss} is sub-Gaussian with mean $0$ and variance proxy $(\sigma')^2 \sum_{i=1}^n \sum_{j=1}^n w_i^2 v_j^2 = (\sigma')^2$. 
Then, by \cref{fact:uni-subgauss}, we have that $\mathbb{P}(\Abs{\langle v, \bm{R} w\rangle} \geq \epsilon \sigma \sqrt{n} / 2) \leq 2 e^{-\Omega(\epsilon^2 n \sigma^2 / (\sigma')^2)}$.

Second, we show that $\Norm{\bm{R}w}_2$ is large.
We start by observing that $\mathbb{E} \Norm{\bm{R}w}_2^2 = w^\top \Paren{\mathbb{E} \bm{R}^\top \bm{R}} w = \sigma^2 n$.
To get a bound on the deviation from the mean, we start by rewriting $\Norm{\bm{R}w}_2^2 = \vectorize(\bm{R})^\top (A) \vectorize(\bm{R})$, where $\vectorize(\bm{R}) \in \R^{n^2}$ has entries $\vectorize(\bm{R})_{(i,j)} = \bm{R}_{i,j}$, and where $A \in \R^{n^2 \times n^2}$ has entries
\[A_{(i_1,j_1),(i_2,j_2)} = \begin{cases}
w_{j_1} w_{j_2} & \text{ if } i_1 = i_2\,,\\
0 & \text{ otherwise}\,.
\end{cases}\]
We remark that, up to a permutation of the rows and columns, $A$ is a block diagonal matrix with $n$ blocks of size $n \times n$, each block equal to $ww^\top$.
From this, a simple calculation shows that $\Norm{A} = 1$ and $\Norm{A}_F^2 = n$.
Then, by the Hanson-Wright inequality in~\cref{fact:hanson-wright},
\[\mathbb{P}\Paren{\Abs{ \Norm{\bm{R}w}_2^2 - \sigma^2 n  } > t} \leq 2 \exp\Paren{-\Omega\Paren{\min\Paren{\frac{t^2}{ (\sigma')^4 n}, \frac{t}{(\sigma')^2}}}}\,.\]
We note that, if a random variable is $(\sigma')^2$-sub-Gaussian, then its variance is bounded by $(\sigma')^2$, so $\sigma^2 \leq (\sigma')^2$.
Then, we get that $\mathbb{P}(\Norm{\bm{R}w}_2^2 < \sigma^2 n / 4) \leq 2e^{-\Omega(n \sigma^4/(\sigma')^4)}$, so\linebreak $\mathbb{P}(\Norm{\bm{R}w}_2 < \sigma \sqrt{n} / 2) \leq 2e^{-\Omega(n \sigma^4/(\sigma')^4)}$.

Therefore, putting the bounds on the numerator and the denominator together,
\[\mathbb{P} \Paren{\Abs{\Iprod{ v, \frac{\bm{R} w}{\Norm{\bm{R} w}_2}} } > \epsilon} \leq 2 e^{-\Omega(\epsilon^2 n \sigma^2 / (\sigma')^2)} + 2e^{-\Omega(n \sigma^4/(\sigma')^4)} \leq e^{-\delta n}\]
for $\delta = \Omega(\epsilon^2 \sigma^4/(\sigma')^4)$ small enough, and then $\bm{X}'$ is $(\epsilon, \delta)$-partially rotated.
Choosing $\epsilon=0.01$ gives the stated result.
\end{proof}

\restatecorollary{cor:main}
\begin{proof}%
If $n \geq C k \log d / \gamma_{\supp(\beta)}(\bm{X}_S)$ for the absolute constant $C > 0$ required by~\Cref{thm:main}, then by applying the theorem with $S'=\supp(\beta)$, we get that $\gamma_{\supp(\beta)}(\bm{X}) \geq 0.9 \gamma_{\supp(\beta)}({\bm{X}}_S)$ with high probability.
Then the desired bound follows by a straightforward adjustment of the proof of Theorem 7.20b in~\cite{MR3967104-Wainwright19} to the constrained Lasso with the definition of the RE constant in~\Cref{def:re}.
Specifically, one would start from Equation 7.28 in~\cite{MR3967104-Wainwright19}, apply \Holder's inequality to the right-hand side, bound $\norm{\hat{\Delta}}_1 \leq 2\sqrt{k} \norm{\hat{\Delta}_{\supp(\beta)}}_2 \leq 2 \sqrt{\frac{k}{\gamma_{\supp(\beta)}(\bm{X}) \cdot n}} \Norm{\bm{X} \hat\Delta}_2$, and finally divide by $\Norm{\bm{X} \hat\Delta}_2$.

Else, if $n < C k \log d / \gamma_{\supp(\beta)}(\bm{X}_S)$, then the desired bound is trivial.
Specifically, the Lasso estimate always satisfies the trivial prediction error bound $O(\norm{\bm{w}}_2^2/n)$, which is with high probability bounded by $O(\sigma^2)$; e.g., this is immediate by applying Cauchy-Schwarz in Equation 7.28 in~\cite{MR3967104-Wainwright19}.
When $n < C k \log d / \gamma_{\supp(\beta)}(\bm{X}_S)$, we simply observe that this bound of $O(\sigma^2)$ already satisfies the conclusion of our corollary.
\end{proof}

\section{RIP implies RNO}
\label{sec:rip-to-rno}

In this section we give a simple proof that RIP implies RNO. Let us first define RIP:

\begin{definition}[Restricted isometry property]
Let $s \in \mathbb{N}$ and $\delta \geq 0$. 
A matrix $X \in \R^{n \times d}$ satisfies the $(s, \delta)$-\emph{restricted isometry property} (RIP) condition if, for all $s$-sparse $\beta \in \R^{d}$,
\[(1-\delta) \norm{\beta}_2 \leq \norm{X \beta}_2 \leq (1+\delta) \norm{\beta}_2\,.\]
\end{definition}

Then we obtain the following implication:

\begin{proposition}[RIP implies RNO]
\label[proposition]{prop:rip-to-rno}
Suppose a matrix $X \in \R^{n \times d}$ satisfies $(2s, \delta)$-RIP.
Then, for any $S \subseteq [d]$, $X$ satisfies $(s, S, \frac{4\delta}{(1-\delta)^2})$-RNO.
\end{proposition}
\begin{proof}
By the definition of RNO, we can ignore the cases $\alpha = 0$ and $\beta = 0$.
Because the inner product in the definition of RNO is invariant to scalings of $\alpha$ and $\beta$, we can also assume that both $\alpha$ and $\beta$ are unit vectors.
Consider then some unit $s$-sparse $\alpha \in \R^{|S|}$ and unit $s$-sparse $\beta \in \R^{|S^c|}$. Then 
\[\Norm{X_S \alpha + X_{S^c} \beta}^2 = \Norm{X_S \alpha}_2^2 + \Norm{X_{S^c} \beta}_2^2 + 2 \Norm{X_S \alpha}_2 \Norm{X_{S^c}\beta}_2 \Iprod{\frac{X_S \alpha}{\Norm{X_S \alpha}_2}, \frac{X_{S^c} \beta}{\Norm{X_{S^c} \beta}_2}}\,,\]
so
\[\Iprod{\frac{X_S \alpha}{\Norm{X_S \alpha}_2}, \frac{X_{S^c} \beta}{\Norm{X_{S^c} \beta}_2}} = \frac{\Norm{X_S \alpha + X_{S^c} \beta}_2^2 - \Norm{X_S \alpha}_2^2 - \Norm{X_{S^c} \beta}_2^2}{2 \Norm{X_S \alpha}_2 \Norm{X_{S^c}\beta}_2 }\,.\]

By $(2s, \delta)$-RIP, we have the following three inequalities:
\[(1-\delta)^2 \Paren{\Norm{\alpha}_2^2 + \Norm{\beta}_2^2} \leq \Norm{X_S \alpha + X_{S^c} \beta}_2^2 \leq (1+\delta)^2 \Paren{\Norm{\alpha}_2^2 + \Norm{\beta}_2^2}\,,\]
\[(1-\delta)^2 \Norm{\alpha}_2^2 \leq \Norm{X_S \alpha}_2^2 \leq (1+\delta)^2 \Norm{\alpha}_2^2\,,\]
\[(1-\delta)^2 \Norm{\beta}_2^2 \leq \Norm{X_{S^c} \beta}_2^2 \leq (1+\delta)^2 \Norm{\beta}_2^2\,.\]
Therefore, 
\[
\Abs{\Iprod{\frac{X_S \alpha}{\Norm{X_S \alpha}_2}, \frac{X_{S^c} \beta}{\Norm{X_{S^c} \beta}_2}}}
\leq \frac{4\delta \norm{\alpha}_2^2 + 4\delta \norm{\beta}_2^2}{2(1-\delta)^2 \norm{\alpha}_2 \norm{\beta}_2}
\leq \frac{4\delta}{(1-\delta)^2}\,.
\]
\end{proof}

\section{Counterexample for alternative definition of RE constant}
\label{sec:strong-re}

In this section we give a counterexample to~\Cref{thm:main} under the alternative definition of the RE constant in~\Cref{eq:alt-re}.
Specifically, we construct a family of matrices $\bm{X} \in \R^{n \times d}$ that are partially-rotated with respect to a subset $S \subseteq [d]$, but for which even for arbitrarily large $n$ it holds that $\gamma'_{S}(\bm{X}) \lesssim \gamma'_{S}(\bm{X}_S) / k$.

\begin{proposition}
\label[proposition]{prop:strong-re}
Let $n, d \geq k + 2$.
Then there exists a matrix $\bm{X} \in \R^{n \times d}$ that is partially-rotated with respect to a set $S \subseteq [d]$ and for which $\gamma'_{S}(\bm{X}_S)=1$ and with high probability $\gamma'_{S}(\bm{X}) \leq O(1/k)$.
\end{proposition}
\begin{proof}
Consider the matrix $A \in \R^{n \times (k+2)}$ with the following columns:
\begin{itemize}
    \item For $i=1,\ldots,k$, let $A_i = \sqrt{n} e_i$, where $e_i$ is the vector with $1$ on coordinate $i$ and $0$ elsewhere,
    \item Let $A_{k+1} = A_{k+2}$ be arbitrary of norm $\sqrt{n}$.
\end{itemize}
Let $\bm{R}^{n \times n}$ be a uniformly random orthogonal matrix, and let $\bm{X} = [A_1, \ldots, A_k, \bm{R} A_{k+1}, \bm{R} A_{k+2}] \in \R^{n \times (k+2)}$.
Denote the set $\{1, \ldots, k\}$ by $[k]$.
Then $\bm{X}$ is partially-rotated with respect to $[k]$.
We also have trivially that $\gamma'_{[k]}(\bm{X}_{[k]}) = 1$.

Consider now the vector $\beta = [1, \ldots, 1, k/2, -k/2] \in \R^{k+2}$.
We have that $\beta \in \mathbb{C}([k])$. 
Furthermore, because $\bm{X}_{k+1}=\bm{X}_{k+2}$, we have that $\bm{X}\beta = \bm{X}_{[k]} \beta_{[k]}$, so $\norm{\bm{X}\beta}^2 = nk$.
Then 
\[\frac{\frac{1}{n}\Norm{\bm{X}\beta}^2}{\Norm{\beta}^2} = \frac{k}{k + k^2/2} \leq O(1/k)\,,\]
so $\gamma'_{[k]}(\bm{X}) \leq O(1/k)$.
Finally, the same conclusion holds if we append additional random columns to $\bm{X}$ until it has $d$ columns (we set $\beta$ to be $0$ on the additional coordinates).
\end{proof}

\end{document}